\def\draft{n}
\documentclass[12pt]{amsart}
\pdfoutput=1
\usepackage[headings]{fullpage}
\usepackage{amssymb,epic,eepic,epsfig}
\usepackage{graphicx}
\usepackage{texdraw}
\usepackage{url}
\usepackage{subfigure}
\usepackage{multirow}
\usepackage{array}
\usepackage[bookmarks=true,%
    colorlinks=true,%
    linkcolor=blue,%
    citecolor=blue,%
    filecolor=blue,%
    menucolor=blue,%
    urlcolor=blue,%
    breaklinks=true]{hyperref}

\usepackage{enumerate}
\usepackage{enumitem}
  
\usepackage{scalerel} 

\usepackage{mathdots}
\usepackage{amsfonts}

\usepackage{caption}
\usepackage{capt-of}

\usepackage[all]{xy}

\newtheorem{theorem}{Theorem}[section]
\theoremstyle{definition}
\newtheorem{proposition}[theorem]{Proposition}
\newtheorem{lemma}[theorem]{Lemma}
\newtheorem{definition}[theorem]{Definition}

\newtheorem{remark}[theorem]{Remark}
\newtheorem{corollary}[theorem]{Corollary}
\newtheorem{conjecture}[theorem]{Conjecture}

\numberwithin{equation}{section}

\def\printname#1{
        \if\draft y
                \smash{\makebox[0pt]{\hspace{-0.5in}
                        \raisebox{8pt}{\tt\tiny #1}}}
        \fi
}

\newlength{\standardunitlength}
\catcode`\@=11
\long\def\@makecaption#1#2{%
     \vskip 10pt

\setbox\@tempboxa\hbox{
       \small\sf{\bfcaptionfont #1. }\ignorespaces #2}%
     \ifdim \wd\@tempboxa >\captionwidth {%
         \rightskip=\@captionmargin\leftskip=\@captionmargin
         \unhbox\@tempboxa\par}%
       \else
         \hbox to\hsize{\hfil\box\@tempboxa\hfil}%
     \fi}
\font\bfcaptionfont=cmssbx10 scaled \magstephalf
\newdimen\@captionmargin\@captionmargin=2\parindent
\newdimen\captionwidth\captionwidth=\hsize
\catcode`\@=12

\newcommand{\Span}{\operatorname{Span}}

\def\cxymatrix#1{\xy*[c]\xybox{\xymatrix#1}\endxy}

\def\Mtildehat{\scalerel*{\widehat{\widetilde M}}{\widetilde{M^2}}}

\def\BC{\mathbb C}

\def\D{\Delta}

\def\a{\alpha}

\def\S{\Sigma}

\def\SL{\mathrm{SL}}
\def\PSL{\mathrm{PSL}}

\DeclareMathOperator{\Conj}{Conj}

\def\be{  \begin{equation} }
\def\ee{  \end{equation} }

\def\diag{\text{diag}}

\def\CT{\mathcal T}

\def\R{\mathbb{R}}
\def\Z{\mathbb{Z}}
\def\C{\mathbb{C}}

\def\Q{\mathbb{Q}}
\def\T{\mathcal{T}}

\def\rd{\mathrm{red}}
\def\Tr{\mathrm{Tr}}
\def\Ker{\mathrm{Ker}}

\def\rank{\mathrm{rank}}
\def\Span{\mathrm{Span}}
\def\Stab{\mathrm{Stab}}
\usepackage{ifpdf}
\ifpdf
\fi

\begin{document}

\begin{abstract}
The Ptolemy coordinates for boundary-unipotent $\SL(n,\C)$-representations of a 
3-manifold group were introduced in 
\cite{GaroufalidisThurstonZickert} inspired by the $\mathcal A$-coordinates 
on higher Teichm\"{u}ller space due to Fock and Goncharov. In this paper, 
we define the Ptolemy field of a (generic) $\PSL(2,\C)$-representation and prove that it coincides with the trace field
of the representation. This gives an efficient algorithm to compute the
trace field of a cusped hyperbolic manifold.
\end{abstract}


\title[The Ptolemy field of $3$-manifold-representations]{
The Ptolemy field of $3$-manifold-representations}
\author{Stavros Garoufalidis}
\address{School of Mathematics \\
         Georgia Institute of Technology \\
         Atlanta, GA 30332-0160, USA \newline
         {\tt \url{http://www.math.gatech.edu/~stavros }}}
\email{stavros@math.gatech.edu}
\author{Matthias Goerner}
\address{University of Maryland \\
         Department of Mathematics \\
         College Park, MD 20742-4015, USA \newline
         {\tt \url{http://math.berkeley.edu/~matthias/}}}
\email{enischte@gmail.com}
\author{Christian K. Zickert}
\address{University of Maryland \\
         Department of Mathematics \\
         College Park, MD 20742-4015, USA \newline
         {\tt \url{http://www2.math.umd.edu/~zickert}}}
\email{zickert@umd.edu}
\thanks{S.~G.~and C.~Z.~were supported in part by the NSF. \\
\newline
1991 {\em Mathematics Classification.} Primary 57N10. Secondary 57M27.
}

\date{\today }

\maketitle

\tableofcontents

\section{Introduction}

\subsection{The Ptolemy coordinates}
\label{sub.ge}

The Ptolemy coordinates for boundary-unipotent representations of a 
3-manifold group in $\SL(n,\C)$ were introduced in 
Garoufalidis--Thurston--Zickert~\cite{GaroufalidisThurstonZickert} inspired by the $\mathcal A$-coordinates 
on higher Teichm\"{u}ller space due to Fock and Goncharov
\cite{FockGoncharov}. In this paper we will focus primarily on representations in $\SL(2,\C)$ and $\PSL(2,\C)$.  

Given a topological ideal triangulation $\T$ of an oriented compact 3-manifold $M$, 
a \emph{Ptolemy assignment} (for $\SL(2,\C)$) is an assignment of a non-zero 
complex number (called a \emph{Ptolemy coordinate}) to each $1$-cell of $\CT$, such that for each simplex, the Ptolemy coordinates assigned to the edges $\varepsilon_{ij}$
satisfy the \emph{Ptolemy relation}
\begin{equation}\label{eq:PtolemyRelation}
c_{03}c_{12}+c_{01}c_{23}=c_{02}c_{13}.
\end{equation}
The set of Ptolemy assignments is thus an affine variety $P_2(\T)$, which is cut out by homogeneous quadratic polynomials. 

In this paper we define the Ptolemy field of a boundary-unipotent representation and show that it is isomorphic to the trace field. This gives rise to an efficient algorithm for \emph{exact} computation of the trace field of a hyperbolic manifold.


 

\subsection{Decorated $\SL(2,\BC)$-representations}
\label{sub.decorated}
The precise relationship between Ptolemy assignments and representations is given by
\begin{equation}\label{eq:summaryof1to1}
\xymatrix{\left\{\txt{Points in\\$P_2(\T)$}\right\}\ar@{<->}[r]^-{1-1}&\left\{\txt{Natural $(\SL(2,\C),P)$-\\cocycles on $M$}\right\}\ar@{<->}[r]^-{1-1}&\left\{\txt{Generically decorated\\
$(\SL(2,\C),P)$-representations}\right\}.}
\end{equation}
The concepts are briefly described below, and the correspondences are illustrated in Figures~\ref{fig:DecToPtolemy} and~\ref{fig:PtolemyToCocycle}. We refer to Section~\ref{sec:Notation} for a summary of our notation.
The bijections of~\eqref{eq:summaryof1to1} first appeared in Zickert~\cite{ZickertDuke} (in a slightly different form), and were generalized to $\SL(n,\C)$-representations in Garoufalidis--Thurston--Zickert~\cite{GaroufalidisThurstonZickert}. 

\begin{itemize}[itemsep=1mm]
\item \textit{Natural cocycle:} Labeling of the edges of each truncated simplex by elements in $\SL(2,\C)$ satisfying the cocycle condition (the product around each face is $1$). The long edges are counter diagonal, i.e.~of the form $\left(\begin{smallmatrix}0&-x^{-1}\\x&0\end{smallmatrix}\right)$, and the short edges are \emph{non-trivial} elements in $P$. Identified edges are labeled by the same group element.
\item \textit{Decorated representation:} A \emph{decoration} of a boundary-parabolic representation $\rho$ is an assignment of a coset $gP$ to each vertex of $\Mtildehat$ which is equivariant with respect to $\rho$. A decoration is \emph{generic} if for each edge joining two vertices, the two $P$-cosets $gP$, $hP$ are distinct as $B$-cosets. This condition is equivalent to $\det(ge_1,he_1)\neq 0$. Two decorations are considered equal if they differ by left multiplication by a group element $g$.
\end{itemize}

\begin{figure}[htb]
\begin{center}
\hspace{-30mm}
\begin{minipage}[c]{0.3\textwidth}
\scalebox{0.65}{\input{figures_gen/BasicPtolemy.tex}}
\end{minipage}
\hspace{-10mm}
\begin{minipage}[c]{0.3\textwidth}
\scalebox{0.65}{\input{figures_gen/TruncatedCocycle.tex}}
\end{minipage}
\begin{minipage}[c]{0.3\textwidth}
\scalebox{0.62}{\input{figures_gen/PtolemyToCocycleShort.tex}}
\end{minipage}
\\
\hspace{-20mm}
\begin{minipage}[t]{0.32\textwidth}
\captionsetup{width=1\textwidth}
\caption{Ptolemy assignment; the Ptolemy relation~\eqref{eq:PtolemyRelation} holds.}\label{fig:Ptolemy}
\end{minipage}
\hspace{-10mm}
\begin{minipage}[t]{0.32\textwidth}
\captionsetup{width=\textwidth}
\caption[Bla]{Natural cocycle; $\alpha$ is counter-diagonal, $\beta\in P$.}\label{fig:Cocycle}
\end{minipage}
\hspace{5mm}
\begin{minipage}[t]{0.32\textwidth}
\captionsetup{width=\textwidth}
\caption[Bla]{From Ptolemy assignments to natural cocycles.}\label{fig:PtolemyToCocycle}
\end{minipage}
\end{center}
\end{figure}

\begin{figure}[htb]
\begin{center}
\hspace{4mm}
\begin{minipage}[c]{0.3\textwidth}
\scalebox{0.7}{\input{figures_gen/Decoration.tex}}
\end{minipage}
\hfill
\begin{minipage}[c]{0.48\textwidth}
\scalebox{0.63}{\input{figures_gen/DecorationToPtolemy.tex}}
\end{minipage}
\hfill
\\
\begin{minipage}[t]{0.48\textwidth}
\captionsetup{width=1\textwidth}
\caption{Decoration; equivariant assignment of cosets.}\label{fig:Decoration}
\end{minipage}
\hspace{-2mm}
\begin{minipage}[t]{0.48\textwidth}
\captionsetup{width=\textwidth}
\caption[Bla]{From decorations to Ptolemy assignments.}\label{fig:DecToPtolemy}
\end{minipage}
\end{center}
\end{figure}

By ignoring the decoration, \eqref{eq:summaryof1to1} yields a map
\begin{equation}\label{eq:PtolemyToRep}
\mathcal R\colon P_2(\T)\to \big\{(\SL(2,\C),P)\text{-representations}\big\}
\big/\Conj.
\end{equation}
The representation corresponding to a Ptolemy assignment is given explicitly in terms of the natural cocycle.

\begin{remark}
Note that a natural cocycle canonically determines a representation of the edge path groupoid of the triangulation of $M$ by truncated simplices.
\end{remark}

\begin{remark} 
Every boundary-parabolic representation has a decoration, but a representation may have only non-generic decorations. 
The map $\mathcal R$ is thus not surjective in general, and the image depends on the triangulation. However, if the triangulation is sufficiently fine, 
$\mathcal R$ is surjective (see~\cite{GaroufalidisThurstonZickert}). The preimage of a representation depends on the image of the peripheral subgroups (see Proposition~\ref{prop:GenericBoundaryNonTrivial}).\end{remark}

\subsection{Obstruction classes and $\PSL(2,\C)$-representations}
There is a subtle distinction between representations in $\SL(2,\C)$ versus $\PSL(2,\C)$. The geometric representation of a hyperbolic manifold always lifts to an $\SL(2,\C)$-representation, but for a one-cusped manifold, no lift is boundary-parabolic (any lift will take a longitude to an element of trace $-2$~\cite{Calegari}). 

The obstruction to lifting a boundary-parabolic $\PSL(2,\C)$-representation to a boundary-parabolic $\SL(2,\C)$-representation is a class in $H^2(\widehat M;\Z/2\Z)$. For each such class there is a Ptolemy variety $P^\sigma_2(\T)$, which maps to the set of $\PSL(2,\C)$-representations with obstruction class $\sigma$. More precisely, $P^\sigma_2(\T)$ is defined for each $2$-cocycle $\sigma\in Z^2(\widehat M;\Z/2\Z)$, and up to canonical isomorphism only depends on the cohomology class of $\sigma$. The Ptolemy variety for the trivial cocycle equals $P_2(\T)$. The analogue of~\eqref{eq:summaryof1to1} is
\begin{equation}\label{eq:1to1Obstruction}
\xymatrix{\left\{\txt{Points in\\$P^{\sigma}_2(\T)$}\right\}\ar@{<->}[r]^-{1-1}&\left\{\txt{Lifted natural\\$(\PSL(2,\C),P)$-cocycles\\with obstruction cocycle $\sigma$}\right\}\ar@{->>}[r]&\left\{\txt{Generically decorated\\
$(\PSL(2,\C),P)$-representations\\with obstruction class $\sigma$}\right\}.}
\end{equation}
A lifted natural cocycle is defined as above, except that the product along a face is now $\pm I$, where the sign is determined by $\sigma$. The right map is no longer a $1$-$1$ correspondence; the preimage of each decorated representation is the choice of lifts, i.e.~parametrized by a cocycle in $Z^1(\widehat M;\Z/2\Z)$. We refer to \cite{GaroufalidisThurstonZickert} for details. As in~\eqref{eq:PtolemyToRep}, ignoring the decoration yields a map
\begin{equation}\label{eq:PtolemyToRepObstruction}
\mathcal R\colon P_2^{\sigma}(\T)\to \left\{\txt{$(\PSL(2,\C),P)$-representations\\with obstruction class $\sigma$}\right\}
\big/\Conj.
\end{equation}
which is explictly given in terms of the natural cocycle.
\begin{theorem}[Garoufalidis--Thurston--Zickert~\cite{GaroufalidisThurstonZickert}]\label{thm:EssentialEdges}
If $M$ is hyperbolic, and all edges of $\T$ are essential, the geometric representation is in the image of $\mathcal R$.
\end{theorem}
\begin{remark}
If $\T$ has a non-essential edge, all Ptolemy varieties will be empty. Hence, if $P^\sigma_2(\T)$ is non-empty for some $\sigma$, and if $M$ is hyperbolic, the geometric representation is detected by the Ptolemy variety of the geometric obstruction class.
\end{remark}

\subsection{Our results}
We view the Ptolemy varieties $P^\sigma_2(\T)$ as subsets of an ambient space $\C^e$, with coordinates indexed by the $1$-cells of $\T$. Let $T=(\C^*)^v$, with the coordinates indexed by the boundary components of $M$. 

\begin{definition} The \emph{diagonal action} is the action of $T$ on $P^\sigma_2(\T)$ where $(x_1,\dots,x_v)\in T$ acts on a Ptolemy assignment by replacing the Ptolemy coordinate $c$ of an edge $e$ with $x_ix_jc$, where $x_i$ and $x_j$ are the coordinates corresponding to the ends of $e$. Let 
\begin{equation}
P^\sigma_2(\T)_\rd=P^\sigma_2(\T)\big/T.
\end{equation}
\end{definition}

\begin{definition}
A boundary-parabolic $\PSL(2,\C)$-representation is \emph{generic} if it has a generic decoration. It is \emph{boundary-non-trivial} if each peripheral subgroup has non-trivial image.
\end{definition}
\begin{remark}
Note that the notion of genericity is with respect to the triangulation. By Theorem~\ref{thm:EssentialEdges}, if all edges of $\T$ are essential (and $\T$ has no interior vertices), the geometric representation of a cusped hyperbolic manifold is always generic and boundary-non-trivial.
\end{remark}
\begin{remark}
Note that if $M$ has spherical boundary components (e.g.~if $\T$ is a triangulation of a closed manifold), no representation is boundary-non-trivial.
\end{remark}

\begin{proposition}\label{prop:GenericBoundaryNonTrivial}
The map $\mathcal R$ in~\eqref{eq:PtolemyToRepObstruction} factors through $P_2^{\sigma}(\T)_{\rd}$, i.e.,~we have
\begin{equation}\label{eq:GenericBoundaryNonTrivial}
\cxymatrix{{\mathcal R\colon P_2^{\sigma}(\T)_{\rd}\ar@{->}[r]& \left\{\txt{$(\PSL(2,\C),P)$-representations\\with obstruction class $\sigma$}\right\}\big/\Conj.}}
\end{equation}
The image is the set of generic representations, and the preimage of a generic, boundary-non-trivial representation is finite and parametrized by $H^1(\widehat M;\Z/2\Z)$. 
\end{proposition}
\begin{remark}
For the corresponding map from $P_2(\T)_\rd$ to $(\SL(2,\C),P)$-representations, the preimage of a generic boundary-non-trivial representation is a single point.
\end{remark}

\begin{remark}\label{rm:HigherComponents}
The preimage of a representation which is not boundary-non-trivial is never finite. In fact its dimension is the number of boundary-components that are collapsed. In particular, it follows that if $c\in P_2^{\sigma}(\T)_{\rd}$ is in a $0$-dimensional component (which is not contained in a higher dimensional component), the image is boundary-non-trivial.
\end{remark}

By geometric invariant theory, $P^\sigma_2(\T)_\rd$ is a variety whose coordinate ring is the ring of invariants $\mathcal O^T$ of the coordinate ring $\mathcal O$ of $P^\sigma_2(T)$.

\begin{definition}\label{def:PtolemyField}
Let $c\in P^\sigma_2(\T)$. The Ptolemy field of $c$ is the field
\begin{equation}
k_c=\Q\left(\big\{p(c_1,\dots,c_e)\mid p\in\mathcal O^T\big\}\right).
\end{equation}
The Ptolemy field of a generic boundary-non-trivial representation is the Ptolemy field of any preimage under~\eqref{eq:GenericBoundaryNonTrivial}.
\end{definition}
Clearly, the Ptolemy field only depends on the image in $P^\sigma_2(\T)_\rd$. Our main result is the following.
\begin{theorem}\label{thm:PtolemyFieldEqualsTraceField}
The Ptolemy field of a boundary-non-trivial, generic, boundary-parabolic representation $\rho$ in $\PSL(2,\C)$ or $\SL(2,\C)$ is equal to its trace field.\qed
\end{theorem}
\begin{remark}
For a cusped hyperbolic $3$-manifold the \emph{shape field} is in general smaller than the trace field. The shape field equals the \emph{invariant trace field} (see, e.g.,~\cite{MaclachlanReid}).
\end{remark}
For computations of the Ptolemy field, we need an explicit description of the ring of invariants $\mathcal O^T$, or equivalently, the reduced Ptolemy variety $P^\sigma_2(\T)_\rd$.
\begin{proposition}\label{prop:FixDecoration}
There exist $1$-cells $\varepsilon_1,\dots,\varepsilon_v$ of $\T$ such that the reduced Ptolemy variety $P_2^{\sigma}(\T)_{\rd}$ is naturally isomorphic to the subvariety of $P^\sigma_2(\T)$ obtained by intersecting with the affine hyperplane $c_{\varepsilon_1}=\dots=c_{\varepsilon_v}=1$. 
\end{proposition}
\begin{corollary}
Let $c\in P^\sigma_2(\T)_\rd$. Under an isomorphism as in Proposition~\ref{prop:FixDecoration}, the Ptolemy field of $c$ is the field generated by the Ptolemy coordinates.\qed 
\end{corollary}
\begin{remark} 
A concrete method for selecting $1$-cells as in Proposition~\ref{prop:FixDecoration} is described in Section~\ref{sec:AlphaStarBasic}.
\end{remark}
Analogues of our results for higher rank Ptolemy varieties are discussed in Section~\ref{sec:HigherPtolemy}. The analogue of Proposition~\ref{prop:GenericBoundaryNonTrivial} holds for representations that are \emph{boundary-non-degenerate} (see Definition~\ref{def:BoundaryNonDeg}), and the analogue of Proposition~\ref{prop:FixDecoration} leads to a simple computation of the Ptolemy field. 
\begin{conjecture}
The Ptolemy field of a boundary-non-degenerate, generic, boundary-unipotent representation $\rho$ in $\SL(n,\C)$ or $\PSL(n,\C)$ is equal to its trace field.\qed
\end{conjecture}
\begin{remark}
The computation of reduced Ptolemy varieties is remarkably efficient using Magma~\cite{Magma}.
For all but a few census manifolds, primary decompositions of the (reduced) Ptolemy varieties $P^\sigma_2(\T)$ can be computed in a fraction of a second on a standard laptop. A database of representations can be found at \url{http://ptolemy.unhyperbolic.org/}~\cite{Unhyperbolic}.
All our tools have been incorporated into SnapPy~\cite{SnapPy} by the second author and the Ptolemy fields obtained through:

\begin{verbatim}
>>> from snappy import Manifold
>>> p=Manifold("m019").ptolemy_variety(2,'all')
>>> p.retrieve_solutions().number_field()
... [[x^4 - 2*x^2 - 3*x - 1], [x^4 + x - 1]]
 \end{verbatim}
 
The number fields are grouped by obstruction class. In this example, we see that the Ptolemy variety for the non-trivial obstruction class has a component with number field $x^4+x-1$, which is the trace field of m019. The above code retrieves a precomputed decomposition of the Ptolemy variety from~\cite{Unhyperbolic}.
In sage or SnapPy with Magma installed, you can use \texttt{p.compute\_solutions().number\_field()} to compute the decomposition.
\end{remark}

\section{Notation}\label{sec:Notation}
\subsection{Triangulations}
Let $M$ be a compact oriented $3$-manifold with (possibly empty) boundary. We refer to the boundary components as \emph{cusps} (although they may not be tori). Let $\widehat M$ be the space obtained from $M$ by collapsing each boundary component to a point.
\begin{definition}
A (concrete) \emph{triangulation} of $M$ is an identification of $\widehat M$ with a space obtained from a collection of simplices by gluing together pairs of faces by affine homeomorphisms. For each simplex $\Delta$ of $\T$ we fix an identification of $\Delta$ with a standard simplex.  
\end{definition}
\begin{remark}
By drilling out disjoint balls if necessary (this does not change the fundamental group), we may assume that the triangulation of $M$ is \emph{ideal}, i.e.~that each $0$-cell corresponds to a boundary-component of $M$. For example, we regard a triangulation of a closed manifold as an ideal triangulation of a manifold with boundary a union of spheres.
\end{remark}
\begin{definition}
A triangulation is \emph{oriented} if the identifications with standard simplices are orientation preserving.
\end{definition}
\begin{remark}
All of the triangulations in the SnapPy censuses \texttt{OrientableCuspedCensus}, \texttt{LinkExteriors} and \texttt{HTLinkExteriors}~\cite{SnapPy} are oriented. Unless otherwise specified we shall assume that our triangulations are oriented.
\end{remark}

A triangulation gives rise to a triangulation of $M$ by truncated simplices, and to a triangulation of $\Mtildehat$.

\subsection{Miscellaneous}
\begin{itemize}
\item The number of vertices, edges, faces, and simplices, of a triangulation $\T$ are denoted by $v$, $e$, $f$, and $s$, respectively.
\item The standard basic vectors in $\Z^k$ are denoted by $e_1,\dots,e_{k}$. 
\item The (oriented) edge of simplex $k$ from vertex $i$ to $j$ is denoted by $\varepsilon_{ij,k}$. 
\item The matrix groups $\left\{\left(\begin{smallmatrix}1&x\\0&1\end{smallmatrix}\right)\right\}$ and $\left\{\left(\begin{smallmatrix}a&x\\0&a^{-1}\end{smallmatrix}\right)\right\}$ are denoted by $P$ and $B$, respectively. The higher rank analogue of $P$ is denoted by $N$.
\item A representation is \emph{boundary-parabolic} if it takes each peripheral subroup to a conjugate of $P$. Such is also called a \emph{$(G,P)$-representation} ($G=\SL(2,\C)$ or $\PSL(2,\C)$). In the higher rank case, such a representation is called boundary-unipotent.
\item A triangulation is \emph{ordered} if $\varepsilon_{ij,k}\sim \varepsilon_{i^\prime j^\prime,k^\prime}$ implies that $i<j\iff i^\prime<j^\prime$.
\end{itemize}

\section{The Ptolemy varieties}\label{sec:PtolemyVariety}
We define the Ptolemy variety for $n=2$ following Garoufalidis--Thurston--Zickert~\cite{GaroufalidisThurstonZickert} (see also Garoufalidis--Goerner--Zickert~\cite{GaroufalidisGoernerZickert}).
\subsection{The $\SL(2,\C)$-Ptolemy variety}
Assign to each oriented edge $\varepsilon_{ij,k}$ of $\Delta_k\in \T$ a \emph{Ptolemy coordinate $c_{ij,k}$}. Consider the affine algebraic set $A$ defined by the \emph{Ptolemy relations}
\begin{equation}\label{eq:PtolemyRelationijk}
c_{03,k}c_{12,k}+c_{01,k}c_{23,k}=c_{02,k}c_{13,k}, \qquad k=1,2,\dots,t,
\end{equation}
the \emph{identification relations}
\begin{equation}\label{eq:IdentificationRelations}
c_{ij,k}=c_{i^\prime j^\prime,k^\prime}\qquad \text{when}\qquad \varepsilon_{ij,k}\sim\varepsilon_{i^\prime j^\prime,k^\prime},
\end{equation}
and the \emph{edge orientation relations} $c_{ij,k}=-c_{ji,k}$. By only considering $i<j$, we shall always eliminate the edge orientation relations.
\begin{definition}\label{def:PtolemyVariety} The \emph{Ptolemy variety} $P_2(\T)$ is the Zariski open subset of $A$ consisting of points with non-zero Ptolemy coordinates.
\end{definition}
\begin{remark} One can concretely obtain $P_2(\T)$ from $A$ by adding a dummy variable $x$ and a dummy relation $x\prod c_{ij,k}=1$.
\end{remark}

\begin{remark}\label{rm:PickRepresentative} We can eliminate the identification relations~\eqref{eq:IdentificationRelations} by selecting a representative for each edge cycle. This gives an embedding of the Ptolemy variety in an ambient space $\C^e$, where it is cut out by $s$ Ptolemy relations, one for each simplex. Note that when all boundary components are tori, $s=e$.
\end{remark}

\subsubsection{The figure-$8$ knot}
Consider the ideal triangulation of the figure-$8$ knot complement shown in Figure~\ref{fig:Fig8Ordered}. The Ptolemy variety $P_2(\T)$ is given by
\begin{equation}\label{eq:Fig8PtolemyUnsimplified}
\begin{aligned}[c]
c_{03,0}c_{12,0}+c_{01,0}c_{23,0}&=c_{02,0}c_{13,0},\\
c_{03,1}c_{12,1}+c_{01,1}c_{23,1}&=c_{02,1}c_{13,1},
\end{aligned}
\qquad 
\begin{aligned}[c]
c_{02,0}&=c_{12,0}=c_{13,0}=c_{01,1}=c_{03,1}=c_{23,1}\\
c_{01,0}&=c_{03,0}=c_{23,0}=c_{02,1}=c_{12,1}=c_{13,1}.
\end{aligned}
\end{equation}
By selecting representatives $\varepsilon_{23,0}$ and $\varepsilon_{13,0}$ for the $2$ edge cycles, $P_2(\T)$ embeds in $\C^2$ where it is given by
\begin{equation}\label{eq:Fig8PtolemyTrivial}
c_{23,0}c_{13,0}+c_{23,0}^2=c_{13,0}^2,\qquad c_{13,0}c_{23,0}+c_{13,0}^2=c_{23,0}^2.
\end{equation}
It follows that $P_2(\T)$ is empty, which is no surprise, since the only boundary-parabolic $\SL(2,\C)$-representations of the figure-$8$ knot are abelian. To detect the geometric representation, we need to consider \emph{obstruction classes} (see Section~\ref{sec:ObstructionClasses} below).

\subsubsection{The figure-$8$ knot sister}
Consider the ideal triangulation of the figure-$8$ knot sister shown in Figure~\ref{fig:Fig8SisOri}. The Ptolemy variety $P_2(\T)$ is given by
\begin{equation}\label{eq:Fig8SisPtolemyUnsimplified}
\begin{aligned}[c]
c_{03,0}c_{12,0}+c_{01,0}c_{23,0}&=c_{02,0}c_{13,0},\\
c_{03,1}c_{12,1}+c_{01,1}c_{23,1}&=c_{02,1}c_{13,1},
\end{aligned}
\qquad 
\begin{aligned}[c]
c_{01,0}=-c_{03,0}=c_{23,0}=-c_{01,1}=c_{03,1}=-c_{23,1}\\
c_{02,0}=-c_{12,0}=c_{13,0}=-c_{02,1}=c_{12,1}=-c_{13,1}
\end{aligned}
\end{equation}
Selecting representatives $\varepsilon_{23,0}$ and $\varepsilon_{13,0}$ for the $2$ edge cycles, $P_2(\T)\in\C^2$ is given by
\begin{equation}\label{eq:Fig8SisPtolemyTrivial}
c_{23,0}c_{13,0}+c_{23,0}^2=c_{13,0}^2,\qquad c_{23,0}c_{13,0}+c_{23,0}^2=c_{13,0}^2.
\end{equation}
This is equivalent to
\begin{equation}
x^2-x-1=0,\qquad x=\frac{c_{13,0}}{c_{23,0}}.
\end{equation}


\begin{figure}[htpb]
\centering
\begin{minipage}[c]{0.48\textwidth}
\scalebox{0.57}{\input{figures_gen/Fig8Triangulation.tex}}
\captionsetup{width=10cm}
\caption{Ordered triangulation of the figure $8$ knot. The signs indicate the non-trivial obstruction class.}\label{fig:Fig8Ordered}
\end{minipage}	
\hspace{2mm}
\begin{minipage}[c]{0.48\textwidth}
\scalebox{0.57}{\input{figures_gen/Fig8SisterTriangulation.tex}}
\captionsetup{width=10cm}
\caption{Oriented triangulation of the figure $8$ knot sister. The signs indicate the non-trivial obstruction class.}\label{fig:Fig8SisOri}
\end{minipage}					
\end{figure}
\begin{remark}
Note that for ordered triangulations, the identification relations~\eqref{eq:IdentificationRelations} do not involve minus signs. The triangulation in Figure~\ref{fig:Fig8Ordered} is not oriented.

\end{remark}
\subsection{Obstruction classes}\label{sec:ObstructionClasses}
Each class in $H^2(\widehat M;\Z/2\Z)$ can be represented by a $\Z/2\Z$-valued $2$-cocycle on $\widehat M$, i.e.~an assignment of a sign to each face of $\T$.  
\begin{definition} Let $\sigma$ be a $\Z/2\Z$-valued $2$-cocycle on $\widehat M$. The \emph{Ptolemy variety} for $\sigma$ is defined as in Definition~\ref{def:PtolemyVariety}, but with the Ptolemy relation replaced by
\begin{equation}\label{eq:PtolemyRelationObstruction}
\sigma_{0,k}\sigma_{3,k}c_{03,k}c_{12,k}+\sigma_{0,k}\sigma_{1,k}c_{01,k}c_{23,k}=\sigma_{0,k}\sigma_{2,k}c_{02,k}c_{13,k},
\end{equation}
where $\sigma_{i,k}$ is the sign of the face of $\Delta_k$ opposite vertex $i$.
\end{definition}
\begin{remark} Multiplying $\sigma$ by a coboundary $\delta(\tau)$ corresponds to multiplying the Ptolemy coordinate of a one-cell $e$ by $\tau(e)$ (see~\cite{GaroufalidisThurstonZickert} for details). Hence, up to canonical isomorphism, the Ptolemy variety $P^\sigma_2(\T)$ only depends on the cohomology class of $\sigma$. The Ptolemy variety $P_2(\T)$ is the Ptolemy variety for the trivial obstruction class.
\end{remark}

\subsubsection{Examples}
In both examples above, $H^2(\widehat M;\Z/2\Z)=\Z/2\Z$, and the non-trivial obstruction class $\sigma$ is indicated in Figures~\ref{fig:Fig8Ordered} and~\ref{fig:Fig8SisOri}.

For the Figure-$8$ knot, $P^{\sigma}_2(\T)$ is given by
\begin{equation}\label{eq:Fig8PtolemyNontrivial}
-c_{23,0}c_{13,0}+c_{23,0}^2=-c_{13,0}^2,\qquad -c_{13,0}c_{23,0}+c_{13,0}^2=-c_{23,0}^2,
\end{equation}
which is equivalent to
\begin{equation}
x^2-x+1=0,\qquad x=\frac{c_{13,0}}{c_{23,0}}.
\end{equation} 
The corresponding representations are the geometric representation and its conjugate.

For the Figure 8 knot sister, the Ptolemy variety becomes
\begin{equation}\label{eq:Fig8SisPtolemyNontrivial}
-c_{23,0}c_{13,0}-c_{23,0}^2=c_{13,0}^2,\qquad -c_{23,0}c_{13,0}-c_{23,0}^2=c_{13,0}^2,
\end{equation}
which is equivalent to
\begin{equation}
x^2+x+1=0,\qquad x=\frac{c_{13,0}}{c_{23,0}}.
\end{equation}

\section{The diagonal action}
Fix an ordering of the $1$-cells of $\T$ and of the cusps of $M$. As mentioned in Remark~\ref{rm:PickRepresentative} the Ptolemy variety can be regarded as a subset of the ambient space $\C^e$.

Let $T=(\C^*)^v$ be a torus whose coordinates are indexed by the cusps of $M$. There is a natural action of $T$ on $P^{\sigma}_2(\T)$ defined as follows:
For $x=(x_1,\dots,x_v)\in T$ and $c=(c_1,\dots c_e)\in P^{\sigma}_2(\T)$ define a Ptolemy assignment $cx$ by
\begin{equation}
(xc)_i=x_jx_kc_i,
\end{equation}
where $j$ and $k$ (possibly $j=k$) are the cusps joined by the $i$th edge cycle. The action is thus determined entirely by the $1$-skeleton of $\widehat M$.

\begin{remark}\label{rm:IntrinsicDiagonalAction} There is a more intrinsic definition of this action in terms of decorations: 
Each vertex of $\Mtildehat$ determines a cusp of $M$, and if $D$ is a decoration taking a vertex $w$ to $gP$, the decoration $xD$ takes $w$ to $g\left(\begin{smallmatrix}x_i&0\\0&x_i^{-1}\end{smallmatrix}\right)P$ where $i$ is the cusp determined by $w$. The fact that the two definitions agree under the one-one correspondence \eqref{eq:1to1Obstruction} is an immediate consequence of the relationship given in Figure~\ref{fig:DecToPtolemy}.
\end{remark}

\subsection{The reduced Ptolemy varieties}
\begin{definition}
The \emph{reduced Ptolemy variety} $P^\sigma_2(\T)_\rd$ is the quotient $P^\sigma_2(\T)\big/T$. 
\end{definition}
Let $\mathcal O$ be the coordinate ring of $P^\sigma_2(\T)$, and let $\mathcal O^T$ be the ring of invariants. By geometric invariant theory, the reduced Ptolemy variety is a variety whose coordinate ring is isomorphic to $\mathcal O^T$.

For $i=0,1$, let $C_i$ denote the free abelian group generated by the \emph{unoriented} $i$-cells of $\widehat M$, and consider the maps (first studied by Neumann~\cite{NeumannComb})
\begin{equation}\label{eq:AlphaAlphaStar}
\alpha\colon C_0\to C_1, \qquad \alpha^*\colon C_1\to C_0,
\end{equation}
where $\alpha$ takes a $0$-cell to a sum of its incident $1$-cells, and $\alpha^*$ takes a $1$-cell to the sum of its end points. The maps $\alpha$ and $\alpha^*$ are dual under the canonical identifications $C_i\cong C_i^*$. Also, $\alpha$ is injective, and $\alpha^*$ has cokernel of order $2$ (see \cite{NeumannComb}.)


The following is an elementary consequence of the definition of the diagonal action.
\begin{lemma}
The diagonal action $P_2^\sigma(\T)$, and the induced action on the coordinate ring $\mathcal O$ of $P_2^\sigma(\T)$ are given, respectively, by
\begin{equation}\label{eq:TActionOnPtolemys}
(xc)_i=(\prod_{j=1}^v x_j^{\alpha_{ij}})c_i, \qquad
x(c^w)=\prod_{j=1}^v x_j^{\alpha^*(w)_j}c^w,
\end{equation}
where $c^w$ is the monomial $c_1^{w_1}\cdots c_e^{w_e}\in\mathcal O$, $w\in\Z^e$.\qed
\end{lemma}

\begin{corollary}\label{cor:RingOfInvs}
Let $w_1,\dots,w_{e-v}$ be a basis for $\Ker(\alpha^*)$. The monomials
$c^{w_1},\dots,c^{w_{e-v}}$ generate $\mathcal O^T$.
\qed
\end{corollary}

\subsubsection{Examples}
Suppose the $1$-skeleton of $\widehat M$ looks like in Figure~\ref{fig:Before} (this is in fact the $1$-skeleton of the census triangulation of the Whitehead link complement.) We have
\begin{equation}
\alpha^*=\begin{pmatrix}2&1&1&0\\0&1&1&2\end{pmatrix}
\end{equation}
and the action of $(x_1,x_2)$ on a Ptolemy assignment $c$ is given in Figure~\ref{fig:After}.
\begin{figure}[htpb]
\centering
\begin{minipage}[c]{0.48\textwidth}
\scalebox{0.85}{\input{figures_gen/PtolemyWhBefore.tex}}
\captionsetup{width=10cm}
\caption{Ptolemy assignment.}\label{fig:Before}
\end{minipage}	
\begin{minipage}[c]{0.48\textwidth}
\scalebox{0.85}{\input{figures_gen/PtolemyWhAfter.tex}}
\captionsetup{width=10cm}
\caption{The diagonal action of $(x_1,x_2)$.}\label{fig:After}
\end{minipage}					
\end{figure}
The kernel of $\alpha^*$ is generated by $(0,-2,0,1)^t$ and $(-1,1,1,0)^t$, so we have
\begin{equation}
\mathcal O^T=\langle c_2^{-2}c_4,c_1^{-1}c_2c_3\rangle.
\end{equation}
Also note that in each of the examples in Section~\ref{sec:PtolemyVariety}, $x\in\mathcal O^T$.

For computations we need a more explicit description of the reduced Ptolemy variety.
\begin{definition}
Let $T\colon \Z^n\to\Z^m$ be a homomorphism. We say that $T$ is \emph{basic} if there exists a subset $J$ of $\{e_1,\dots,e_n\}$ such that $T$ maps $\Span(J)$ isomorphically onto the image of $T$. Elements of such a set $J$ are called \emph{basic generators} for $T$.
\end{definition}
 We identify $C_1$ and $C_0$ with $\Z^e$ and $\Z^v$, respectively.
\begin{proposition}\label{prop:Basic}
The map $\alpha^*\colon C_1\to C_0$ is basic.\qed
\end{proposition}
The proof will be relegated to Section~\ref{sec:AlphaStarBasic}, where we shall also give explicit basic generators.
\begin{proposition}\label{prop:ReducedPtolemyVariety}
Let $\varepsilon_{i_1},\dots,\varepsilon_{i_v}$ be basic generators for $\alpha^*$. 
The ring of invariants $\mathcal O^T$ is isomorphic to $\C[c_1,\dots,c_e]$ modulo the Ptolemy relations and the relations $c_{i_1}=\cdots=c_{i_v}=1$, i.e.~the reduced Ptolemy variety is isomorphic to the subset of $P_2^\sigma(\T)$ where the Ptolemy coordinates of the basic generators are $1$.
\end{proposition}
\begin{proof}
Let $w_1,\dots,w_{e-v}$ be a basis for $\Ker(\alpha^*)$. Hence, $w_1,\dots,w_{e-v}$ and $\varepsilon_{i_1},\dots,\varepsilon_{i_v}$ generate $C_1$. We can thus uniquely express each $c_i$ as a monomial in the $w_j$'s and the $c_{i_j}$'s. The result now follows from Corollary~\ref{cor:RingOfInvs}.
\end{proof}

\begin{remark}
This is how the Ptolemy varieties are computed in SnapPy.
\end{remark}
\subsection{Shapes and gluing equations}
One can assign to each simplex a \emph{shape}
\begin{equation}
z=\sigma_3\sigma_2\frac{c_{03}c_{12}}{c_{02}c_{13}}\in\C\setminus\{0,1\},
\end{equation}
and one can show (see \cite{GaroufalidisThurstonZickert,GaroufalidisGoernerZickert}) that these satisfy Thurston's gluing equations.
For the geometric representation of a cusped hyperbolic manifold, the shape field (field generated by the shapes) is equal to the invariant trace field, which is in general smaller than the trace field, see Maclachlan--Reid~\cite{MaclachlanReid}.
\begin{remark}
Note that the shapes are elements in $\mathcal O^T$.
\end{remark}

\subsection{Proof that $\alpha^*$ is basic}\label{sec:AlphaStarBasic}
Since $\alpha^*$ has cokernel of order $2$, it is enough to prove that there is a set of columns of $\alpha^*$ forming a matrix with determinant $\pm 2$.
Recall that the columns of $\alpha^*$ corresponds to $1$-cells of $\T$. We shall thus consider graphs in the $1$-skeleton of $\widehat M$. We recall some basic results from graph theory. All graphs are assumed to be connected.
\begin{definition}
The \emph{incidence matrix} of a graph $G$ with vertices $v_1,\dots,v_k$ and edges $\varepsilon_1,\dots,\varepsilon_l$ is the $k\times l$ matrix $I_G$
whose $ij$th entry is $1$ if $v_i$ is incident to $\varepsilon_j$, and $0$ otherwise.
\end{definition}

\begin{lemma}\label{lemma:IncidenceMatrix}
The rank of $I_G$ is $k-1$. If $G$ is a tree, $I_G$ is a $k\times (k-1)$ matrix, and removing any row gives a matrix with determinant $\pm 1$.\qed
\end{lemma}

\subsubsection{Case $1$: A single cusp}
In this case the result is trivial. The matrix representation for $\alpha^*$ is $\begin{pmatrix}2&\cdots&2\end{pmatrix}$.

\subsubsection{Case $2$: Multiple cusps, self edges}
Suppose $\widehat M$ has a self edge $\varepsilon_1$ (an edge joining a cusp to itself), and consider the graph $G$ consisting of the union of $\varepsilon_1$ with a maximal tree $T$ (see Figure~\ref{fig:TreeWithCycle}). The columns of $\alpha^*$ corresponding to the edges of $G$ then form the matrix 
\begin{equation}
B=\left(\begin{array}{c|c}
2& \\\cline{1-1}
 0 &\raisebox{10pt}{{$I_T$}}
\end{array}\right)
\end{equation}

which, by Lemma~\ref{lemma:IncidenceMatrix}, has determinant $\pm 2$.
\subsubsection{Case $3$: Multiple cusps, no self edges}
Pick a face with edges $\varepsilon_1,\varepsilon_2,\varepsilon_3$, and add edges to form a graph $G$ such that $G\setminus \varepsilon_1$ is a maximal tree (see Figure~\ref{fig:TreeWithThreeCycle}). The corresponding columns form the matrix

\begin{equation}
C=I_G=\left(\begin{array}{c|c}
1&\\ 
0&I_T\\
1&\\\cline{1-1} 
 0&
\end{array}\right)
\end{equation}
which, by Lemma~\ref{lemma:IncidenceMatrix}, has determinant $\pm 2$.
This concludes the proof that $\alpha^*$ is basic.

Note that 
\begin{equation}
\det(B)=\det\begin{pmatrix}2&1\\&1\end{pmatrix}=2,\qquad \det(C)=\det\begin{pmatrix}1&1&\\0&1&1\\1&&1\end{pmatrix}=2,
\end{equation}
i.e.~only the edges and vertices shown in Figures~\ref{fig:TreeWithCycle} and \ref{fig:TreeWithThreeCycle} contribute to the determinant.
\begin{figure}[htb]
\centering
\begin{minipage}[c]{0.45\textwidth}
\centering
\scalebox{1}{\input{figures_gen/TreeWithCycle.tex}}
\end{minipage}
\hfill
\begin{minipage}[c]{0.45\textwidth}
\centering
\input{figures_gen/TreeWithThreeCycle.tex}
\end{minipage}
\\
\begin{minipage}[t]{0.45\textwidth}
\caption{Tree $G$ with $1$-cycle; $G\setminus\varepsilon_1$ is a maximal tree.}\label{fig:TreeWithCycle}
\end{minipage}
\begin{minipage}[t]{0.45\textwidth}
\caption{Tree $G$ with $3$-cycle; $G\setminus\varepsilon_1$ is a maximal tree.}\label{fig:TreeWithThreeCycle}
\end{minipage}
\end{figure}

\begin{remark}
Trees with $1$- or $3$-cycles are also used in~\cite[Sec.4.6]{GHRS} to study index structures.
\end{remark}

\section{The Ptolemy field and the trace field}
\subsection{Explicit description of the Ptolemy field}
By Proposition~\ref{prop:ReducedPtolemyVariety} any $c\in P^\sigma_2(\T)$ is equivalent to a Ptolemy assignment $c^\prime$ whose coordinates for a set of basic generators $\varepsilon_{i_1},\dots,\varepsilon_{i_v}$ is $1$. In particular, it follows that the Ptolemy field (see Definition~\ref{def:PtolemyField}) of $c\in P^\sigma_2(\T)$ is given by
\begin{equation}\label{eq:kc}
k_c=k_{c^{\prime}}=\Q\big(\big\{c_{\varepsilon_1},\dots,c_{\varepsilon_e}\big\}\big).
\end{equation}

\begin{definition}
Let $\rho:\pi_1(M)\to\PSL(2,\C)$ be a representation. The \emph{trace field} of $\rho$ is the field generated by the traces of elements in the image. 
We denote it $k_\rho$.
\end{definition}
Our main result is the following. We defer the proof to Section~\ref{sec:ProofOfMainThm}.
\begin{theorem}\label{thm:PtolemyEqTrace} Let $c\in P^\sigma_2(\T)_\rd$. If the corresponding generic representation $\rho$ of $\pi_1(M)$ in $\PSL(2,\C)$ is boundary-non-trivial, the Ptolemy field of $c$ equals the trace field of $\rho$.\qed
\end{theorem}
\begin{remark}
Note that if $c\in P^\sigma_2(\T)_\rd$ is in a degree $0$ component, the Ptolemy field is a number field. 
\end{remark}
\subsection{The setup of the proof}
Since the natural cocycle is given in terms of the Ptolemy coordinates, it follows that $\rho$ is defined over the Ptolemy field. Hence, the trace field is a subfield of the Ptolemy field.

Fix a maximal tree $G$ with $1$ or $3$-cycle as in Figures~\ref{fig:TreeWithCycle} or \ref{fig:TreeWithThreeCycle}. As explained in Section~\ref{sec:AlphaStarBasic} the edges of $G$ are basic generators of $\alpha^*$. We may thus assume without loss of generality that the Ptolemy coordinates $c_i$ of the edges $\varepsilon_i$ of $G$ are $1$. By~\eqref{eq:kc}, it is thus enough to show that the Ptolemy coordinates of the remaining $1$-cells are in the trace field.

Let $\gamma$ denote the (lifted) natural cocycle of $c$. Then $\gamma$ assigns to each edge path $p$ in $\widehat M$ a matrix $\gamma(p)\in\SL(2,\C)$. Let
\begin{equation}\label{eq:AlphaBeta}
\alpha(a)=\begin{pmatrix}0&-a^{-1}\\a&0\end{pmatrix},\qquad \beta(b)=\begin{pmatrix}1&b\\0&1\end{pmatrix}.
\end{equation}
As shown in Figure~\ref{fig:PtolemyToCocycle}, $\gamma$ takes long and short edges to elements of the form $\alpha(a)$ and $\beta(b)$, respectively, where $a$ and $b$ are given in terms of the Ptolemy coordinates.

Since $\rho$ is boundary-non trivial, there exists for each cusp $i$ of $M$ a peripheral loop $M_i$ with $\gamma(M_i)\in P$ non-trivial. We shall here refer to such loops as \emph{non-trivial}. Fix such non-trivial loops $M_i$, once and for all, and let $m_i\neq 0$ be such that $\gamma(M_i)=\beta(m_i)$.
For any edge path $p$ with end point on a cusp $i$ we can alter $M_i$ by a conjugation if necessary (this does not change $m_i$) so that $p$ is composable with $M_i$.

\subsection{Proof for one cusp}
We first prove Theorem~\ref{thm:PtolemyEqTrace} in the case where there is only one cusp. In this case, all edges are self edges, and $T$ consists of a single edge $\varepsilon_1$.
\begin{lemma}\label{lemma:BootStrapEdgeOneCusp} For any self edge $\varepsilon$, we have $m_1c_\varepsilon\in k_\rho$. 
\end{lemma}
\begin{proof}
Let $X_1$ be a peripheral path such that $X_1\varepsilon$ is a loop (see Figure~\ref{fig:SelfEdgeCase}), and let $x_1$ be such that $\gamma(X_1)=\beta(x_1)$. We have
\begin{equation}\label{eq:OneCuspTraceRel}
\Tr(\gamma(X_1\varepsilon))=\Tr(\beta(x_1)\alpha(c_\varepsilon)) =\Tr\left(\begin{pmatrix}1&x_1\\0&1\end{pmatrix}\begin{pmatrix}0&-c_\varepsilon^{-1}\\c_\varepsilon&0\end{pmatrix}\right)=x_1c_\varepsilon\in k_\rho.
\end{equation}
Applying the same computation to the loop $X_1 M_1 \varepsilon$ yields
\begin{equation}
\Tr(\beta(x_1)\beta(m_1)\alpha(c_\varepsilon))=(x_1+m_1)c_\varepsilon\in k_\rho,
\end{equation}
and the result follows.
\end{proof}
Since the Ptolemy coordinate of $\varepsilon_1$ is $1$, it follows that $m_1\in k_\rho$. Since all edges are self edges, we have $c_\varepsilon\in k_\rho$ for all $1$-cells $\varepsilon$. This concludes the proof in the one cusped case.

\begin{figure}[htb]
\begin{center}
\hspace{-10mm}
\begin{minipage}[c]{0.3\textwidth}
\scalebox{0.8}{\input{figures_gen/SelfEdgeBootStrap.tex}}
\end{minipage}
\hspace{2mm}
\begin{minipage}[c]{0.6\textwidth}
\scalebox{0.8}{\input{figures_gen/TwoCuspEdge.tex}}
\end{minipage}
\hfill
\\
\begin{minipage}[t]{0.32\textwidth}
\captionsetup{width=1\textwidth}
\caption{Self edge.}\label{fig:SelfEdgeCase}
\end{minipage}
\hspace{2mm}
\begin{minipage}[t]{0.64\textwidth}
\captionsetup{width=\textwidth}
\caption[Bla]{Edge between cusps.}\label{fig:TwoCuspCase}
\end{minipage}
\end{center}
\end{figure}

\subsection{The general case} \label{sec:ProofOfMainThm}The general case follows the same strategy, but is more complicated since it involves edge paths between multiple cusps. 

\begin{lemma}\label{lemma:BootStrapEdge}
If $\varepsilon$ is a self edge from cusp $i$ to itself, $m_ic_\varepsilon\in k_\rho$
\end{lemma}
\begin{proof}
The proof is identical to that of Lemma~\ref{lemma:BootStrapEdgeOneCusp}.
\end{proof}

\begin{lemma}\label{lemma:connectTwoCusp}
If two (distinct) cusps $i$ and $j$ are joined by an edge $\varepsilon$ in $G$, we have
\begin{equation}m_i m_j\in k_\rho.
\end{equation} 
\end{lemma}

\begin{proof}
Consider the loop $\varepsilon_j\bar{M}_j\bar{\varepsilon_j}M_i$ shown in Figure~\ref{fig:TwoCuspCase}. A simple computation shows that
\begin{equation}
\Tr(\alpha(c_\varepsilon)\beta(-m_{j})\alpha(-c_{\varepsilon})\beta(m_i))=2+m_i m_{j} c_{\varepsilon}^2.
\end{equation}
Since $\varepsilon\in T$, $c_\varepsilon=1$, and the result follows.
\end{proof}
More generally, the following holds.
\begin{lemma}\label{lemma:AllMiInKRho}
We have $m_i\in k_\rho$ for all cusps $i$.
\end{lemma}

\begin{proof}
If $G$ is a tree with 1-cycle, then $c_1=1$, so Lemma~\ref{lemma:BootStrapEdge} implies that $m_1\in k_\rho$. Inductively applying Lemma~\ref{lemma:connectTwoCusp} for the edge $\varepsilon_j$ connecting cusp $i=j-1$ and $j$ implies the result. If $G$ is a tree with $3$-cycle, the Ptolemy coordinates $c_1,c_2$ and $c_3$ are $1$, so the edges of the face are labeled by $\alpha(1)$ and $\beta(-1)$ only (see Figure~\ref{fig:PtolemyToCocycle}). Inserting the peripheral loops $M_i$ as in Figure~\ref{fig:TreeWithThreeCycleCase}, we obtain
\begin{equation}
\Tr\big(\beta(-1)\beta(m_1)\alpha(1)\beta(-1)\beta(m_2)\alpha(1)\beta(-1)\beta(m_3)\alpha(1)\big)\in k_\rho
\end{equation}
By an elementary computation, the trace equals
\begin{equation}
m_1m_2m_3-m_1m_2-m_2m_3-m_3m_1+2\in k_\rho.
\end{equation}
By Lemma~\ref{lemma:connectTwoCusp}, $m_im_j\in k_\rho$, so $m_1\in k_\rho$. The result now follows as above by inductively applying Lemma~\ref{lemma:connectTwoCusp}.
\end{proof}

Let $\varepsilon$ be an arbitrary $1$-cell. If $\varepsilon$ is a self edge, Lemmas~\ref{lemma:BootStrapEdge} and \ref{lemma:AllMiInKRho} imply that $c_\varepsilon\in k_\rho$. Otherwise, there exists an edge path $p$ in the maximal tree $G\setminus\varepsilon_1$, such that $p*\varepsilon$ is a loop in $\widehat M$. By relabeling the cusps and edges if necessary, we may assume that $p=\varepsilon_{i+1}\!*\!\varepsilon_{i+2}\!*\cdots\!*\!\varepsilon_{j}$, where $\varepsilon_{k}$ goes from cusp $k-1$ to cusp $k$. Pick peripheral paths $X_{k}$ on cusp $k$ connecting the ends (in $M$, not $\widehat M$) of edges $\varepsilon_{k}$ and $\varepsilon_{k+1}$ (see Figure~\ref{fig:ArbitraryEdgeCase}). We obtain a loop that can be composed with arbitrary powers of the peripheral loops $M_{i}, \dots, M_{j}$. We thus obtain the following traces (where $b_k\in\Z$)
\begin{equation}
\Tr(\beta(x_{i}+b_{i}m_{i})\alpha(c_{i+1})\beta(x_{i+1}+b_{i+1}m_{i+1})\alpha(c_{i+2})\cdots \beta(x_{j}+b_jm_j)\alpha(c_\varepsilon))\in k_\rho.
\end{equation}

\begin{figure}[htb]
\begin{center}
\hspace{-10mm}
\begin{minipage}[c]{0.4\textwidth}
\scalebox{0.8}{\input{figures_gen/ThreeCuspCycle.tex}}
\end{minipage}
\begin{minipage}[c]{0.6\textwidth}
\scalebox{0.8}{\input{figures_gen/ManyCuspCase.tex}}
\end{minipage}
\hfill
\\
\begin{minipage}[t]{0.35\textwidth}
\captionsetup{width=1\textwidth}
\caption{3-cycle case.}\label{fig:TreeWithThreeCycleCase}
\end{minipage}
\hspace{2mm}
\begin{minipage}[t]{0.5\textwidth}
\captionsetup{width=\textwidth}
\caption[Bla]{Arbitrary edge $\varepsilon$.}\label{fig:ArbitraryEdgeCase}
\end{minipage}
\end{center}
\end{figure}

It will be convenient to regard $\Tr(\beta(x_i)\alpha(c_{i+1})\beta(x_{i+1})\alpha(c_{i+2})\cdots \beta(x_j)\alpha(c_\varepsilon))$ as a function of variables $x_i$ (disregarding that the $x_i$'s are fixed expressions of the Ptolemy coordinates).
\begin{definition}
Given a function $f(x_1,\dots,x_r)$, let $\Delta_{i}f$ be the function given by
\begin{equation}
\Delta_{i}f(h)=f(x_1,\dots,x_i+h,\dots,x_r)-f(x_1,\dots,x_i,\dots,x_r).
\end{equation}
\end{definition}
The following is elementary.
\begin{lemma}\label{lemma:Deltas} If $f(x_1,\dots,x_r)$ is a polynomial where the exponents of all variables $x_i$ are $0$ or $1$, and where the highest degree term is $ax_1x_2\cdots x_r$, we have
\begin{equation}\label{eq:Deltas}
\Delta_r\big(\cdots\Delta_2\big(\Delta_1f(h_1)\big)(h_2)\cdots\big)=ah_1h_2\dots h_r,
\end{equation}
and is thus independent of the $x_i$'s.\qed
\end{lemma}
If, e.g.,~$f(x_1,x_2)=x_1x_2$, we have
\begin{equation}\label{eq:DeltaExample}
\Delta_1f(h_1)=(x_1+h_1)x_2-x_1x_2=h_1x_2,\quad \Delta_2\big(\Delta_1f(h_1)\big)(h_2)=h_1(x_2+h_2)-h_1x_2=h_1h_2.
\end{equation}

\begin{lemma}\label{lemma:TraceFormula} Let $x_1,\dots,x_r$ be variables and $y_1,\dots,y_r$ be constants.
The expression
\begin{equation}
\Tr\big(\beta(x_1)\alpha(y_1)\cdots\beta(x_r)\alpha(y_r)\big)
\end{equation}
is a polynomial in the $x_i$'s whose unique highest degree term is $\prod_{i=1}^ry_i\prod_{i=1}^rx_i$. Moreover, for each monomial term, the exponent of each variable is either $1$ or $0$.
\end{lemma}
\begin{proof}
This follows by induction on $r$.
\end{proof}

Applying Lemmas \ref{lemma:TraceFormula} and \ref{lemma:Deltas} to the function
\begin{equation}
f(x_i,\dots,x_j)=\Tr\big(\beta(x_i)\alpha(c_{i+1})\beta(x_{i+1})\alpha(c_{i+2})\cdots \beta(x_j)\alpha(c_\varepsilon)\big),
\end{equation} 
we obtain 
\begin{equation}
(m_im_{i+1}\cdots m_jc_ic_{i+1}\cdots c_j)c_\varepsilon\in k_\rho.
\end{equation}
Since all $m_i$'s are in $k_\rho$ by Lemma~\ref{lemma:AllMiInKRho}, and all $c_i$'s are $1$ (since $\varepsilon_i\in T$), it follows that $c_\varepsilon$ is in $k_\rho$. This concludes the proof.

\subsection{Proof of Proposition~\ref{prop:GenericBoundaryNonTrivial}}\label{sec:PtolemyToRepProof}
The fact that $\mathcal R$ factors, follows from the fact that the diagonal action only changes the decoration (by diagonal elements, c.f.~Remark~\ref{rm:IntrinsicDiagonalAction}), not the representation.
Since the preimage of the right map in~\eqref{eq:1to1Obstruction} is parametrized by choices of lifts, i.e.~elements in $Z^1(\widehat M;\Z/2\Z)$, all that remains is to show that the only freedom in the choice of decoration of a boundary-non-trivial representation is the diagonal action. 
This follows from results in ~\cite{GaroufalidisThurstonZickert}: A decoration is an equivariant map
\begin{equation}
D\colon \widehat{\widetilde{M}}^{(0)}\to \PSL(2,\C)\big/P,
\end{equation}
and is thus determined by its image of lifts $\widetilde e_1,\dots,\widetilde e_v$ of the cusps of $M$. The freedom in the choice of $D(\widetilde e_i)$ is the choice of a coset $gP$ satisfying that $g\rho(\Stab(\widetilde e_i))g^{-1}\subset P$, where $\Stab(\widetilde e_i)\subset\pi_1(M)$ is the stabilizer of $\widetilde e_i$, i.e.~a peripheral subgroup corresponding to cusp $i$. Hence, if $\rho(\Stab(\widetilde e_i))$ is non-trivial, the freedom is right multiplication by a diagonal matrix (if it is trivial, any coset works). Hence, if $\rho$ is boundary-non-trivial, the only freedom in choosing a decoration is the diagonal action. 


\section{Ptolemy varieties for $n>2$}\label{sec:HigherPtolemy}
Many of our results generalize in a straightforward way to the higher rank Ptolemy varieties $P_n(\T)$. We recall the definition of these below, and refer to~\cite{GaroufalidisThurstonZickert,GaroufalidisGoernerZickert} for details.

We identify all simplices of $\T$ with a standard simplex
\begin{equation}
\Delta^3_n = \left\{(x_0,x_1,x_2,x_3) \in \R^4\bigm\vert 0\leq x_i \leq n,\enspace x_0 + x_1 + x_2 + x_3 = n\right\}
\end{equation}
and regard $\widehat M$ as a quotient of a disjoint union $\coprod_{k=1}^s\Delta^3_{n,k}$, with a copy $\Delta^3_{n,k}$ of $\Delta^3_n$ for each simplex $k$ of $\T$. Define $\Delta^3_n(\Z)=\Delta^3_n\cap\Z^4$ and $\dot\Delta^3_n(\Z)$ to be $\Delta^3_n(\Z)$ with the four vertex points removed.
A point in $M$ in the image of $\coprod_{k=1}^s\dot\Delta^3_{n,k}(\Z)$ is called an \emph{integral point} of $M$.

\subsection{Definition of the Ptolemy variety}
Assign to each $(t,k)\in\Delta^3_{n,k}(\Z)$ a \emph{Ptolemy coordinate} $c_{t,k}$.
For each simplex $k$, we have $\vert\Delta_{n-2}(\Z)\vert=\binom{n+1}{3}$ \emph{Ptolemy relations}
\begin{equation}
c_{\alpha+1001,k}c_{\alpha+0110,k}+c_{\alpha+1100,k}c_{\alpha+0011,k}=c_{\alpha+1010,k}c_{\alpha+0101,k},
\qquad \alpha\in\Delta_{n-2}(\Z)
\end{equation}
as well as \emph{identification relations},
\begin{equation}\label{eq:IdentificationRelationsN}
c_{t,k}=\pm c_{t^{\prime},k^{\prime}},\qquad\text{when}\qquad (t,k)\sim (t^{\prime},k^{\prime}).
\end{equation}
\begin{remark}The signs in~\ref{eq:IdentificationRelationsN} depend in a non-trivial way on the face pairings (see~\cite{GaroufalidisGoernerZickert}). For ordered triangulations the signs are always positive. As in Remark~\ref{rm:PickRepresentative} we can eliminate the identification relations by selecting a representative of each integral point of $M$.
\end{remark}
\begin{definition} The \emph{Ptolemy variety} $P_n(\T)$ is the subset of the affine algebraic set defined by the Ptolemy- and identification relations, consisting of the points where all Ptolemy coordinates are non-zero.
\end{definition}
For general $n$, we denote the group of upper triangular matrices with $1$ on the diagonal by $N$ (instead of $P$). As in~\ref{eq:summaryof1to1} we have 
\begin{equation}\label{eq:summaryof1to1GeneralN}
\xymatrix{\left\{\txt{Points in\\$P_n(\T)$}\right\}\ar@{<->}[r]^-{1-1}&\left\{\txt{Natural $(\SL(n,\C),N)$-\\cocycles on $M$}\right\}\ar@{<->}[r]^-{1-1}&\left\{\txt{Generically decorated\\
$(\SL(n,\C),N)$-representations}\right\}.}
\end{equation}

\subsection{The diagonal action}
Let $D$ be the group of diagonal matrices in $\SL(n,\C)$. We identify $D$ with the torus $(\C^*)^{n-1}$ via the identification
\begin{equation}
(\C^*)^{n-1}\to D,\qquad (a_1,\dots,a_{n-1})\mapsto \diag(a_1,a_2/a_1,\dots,a_{n-1}/a_{n-2})
\end{equation}
As in Remark~\ref{rm:IntrinsicDiagonalAction}, we have a diagonal action of the torus $T=D^{v}$ on the set of decorated representations, where $(D_1,\dots,D_v)\in T$ acts by replacing the coset $gN$ assigned to a vertex $w$ by $gD_iN$, where $i$ is the cusp corresponding to $w$. The corresponding action on $P_n(\T)$ is described in Lemma~\ref{lemma:ActionOnPtolemy} below.

Let $C_1^n$ be the group generated by the integral points of $M$, and let $C^n_0=C_0\otimes \Z^{n-1}$. In Garoufalidis--Zickert~\cite{GaroufalidisZickert} we defined maps
\begin{equation}
\alpha\colon C^n_0\to C^n_1,\qquad \alpha^*\colon C^n_1\to C^n_0
\end{equation}
generalizing~\eqref{eq:AlphaAlphaStar}. The map $\alpha^*$ takes an integral point $(t,k)$ to $\sum x_i\otimes e_{t_i}$, where $x_i$ is the cusp determined by vertex $i$ of simplex $k$. We shall not need the definition of $\alpha$.
\begin{lemma}[Garoufalidis--Zickert~\cite{GaroufalidisZickert}]\label{lemma:CokernelN}
The map $\alpha^*$ is surjective with cokernel $\Z/n\Z$.
\end{lemma}
By selecting an ordering of the natural generators of $C^n_0$ and $C^n_1$, we regard $\alpha$ and $\alpha^*$ as matrices. The following is an elementary consequence of~\eqref{eq:summaryof1to1GeneralN}.
\begin{lemma}\label{lemma:ActionOnPtolemy}
The diagonal action of $T=(\C^*)^{v(n-1)}$ on $P_n(\T)$ and the corresponding action on the coordinate ring $\mathcal O$ of $P_n(\T)$ are given, respectively, by 
\begin{equation}
\pushQED{\qed}
(xc)_t=(\prod_{j=1}^{v(n-1)} x_j^{\alpha_{ij}})c_t,\qquad x(c^w)=\prod_{j=1}^{v(n-1)} x_j^{\alpha^*(w)_j}c^w.\qedhere
\popQED
\end{equation}
\end{lemma}
\begin{corollary} The ring of invariants $\mathcal O^T$ is generated by $c^{w_1},\dots,c^{w_r}$, where $w_1,\dots,w_r$ are a basis for $\Ker(\alpha^*)$, and $r=\rank(C^n_1)-\rank(C^n_0)$.
\end{corollary}
\begin{definition}
The \emph{Ptolemy field} of a Ptolemy assignment $c\in P_n(\T)$ is defined as
\begin{equation}
k_c=\Q(c^{w_1},\dots,c^{w_r})
\end{equation}
where $w_1,\dots w_r$ are (integral) generators of $\Ker(\alpha^*)$. 
\end{definition}
The following is proved in Section~\ref{sec:AlphaStarBasicGeneral}.
\begin{proposition}
The map $\alpha^*\colon C^n_1\to C^n_0$ is basic.\qed
\end{proposition}
\begin{corollary}
Let $p_1,\dots,p_{(n-1)v}$ be integral points that are basic generators of $C^n_1$. The ring $\mathcal O^T$ is generated by the Ptolemy relations together with the relations $c_{p_1}=\cdots=c_{p_{(n-1)v}}=1$. Equivalently, the reduced Ptolemy variety is isomorphic to the subvariety of $P_n(\T)$ consisting of Ptolemy assignments with $c_{p_i}=1$.
\end{corollary}
\begin{proof}
This follows the proof of Proposition~\ref{prop:ReducedPtolemyVariety} word by word.
\end{proof}
\begin{remark}
This is how the Ptolemy varieties and Ptolemy fields at~\cite{Unhyperbolic} are computed.
\end{remark}
\subsection{Representations}
\begin{definition}\label{def:BoundaryNonDeg}
Let $\rho$ be an $(\SL(n,\C),N)$-representation, and let $I_i$ denote the image of the peripheral subgroup corresponding to cusp $i$. We say that $\rho$ is \emph{boundary-non-degenerate} if each $I_i$ has an element whose Jordan canonical form has a single (maximal) Jordan block.
\end{definition}
\begin{proposition}
The map 
\begin{equation}
\mathcal R\colon P_n(\T)_\rd\to \big\{(\SL(n,\C),N)\text{-representations}\big\}
\big/\Conj.
\end{equation}
maps onto the generic representations, and the preimage of a generic boundary-non-degenerate representation consists of a single point.
\end{proposition}
\begin{proof}
The proof is identical to the proof in Section~\ref{sec:PtolemyToRepProof} for $n=2$.
\end{proof}
\begin{conjecture}
The Ptolemy field of a generic, boundary-non-degenerate representation is equal to its trace field.
\end{conjecture}

\begin{remark}Much of the theory also works for $\PSL(n,\C)$-representations by means of obstruction classes in $H^2(\widehat M;\Z/n\Z)$. When $n$ is even, obstruction classes in $H^2(\widehat M;\Z/2\Z)$ were defined in~\cite{GaroufalidisThurstonZickert} for representations in $p\SL(n,\C)=\SL(n,\C)/\pm I$. 
For $\PSL(n,\C)$-representations, both the Ptolemy field and the trace field are only defined up to $n$th roots of unity.
The generalized obstruction classes are used on the website~\cite{Unhyperbolic} and will be explained in a forthcoming publication. 
\end{remark}

\subsection{Proof that $\alpha^*$ is basic}\label{sec:AlphaStarBasicGeneral}
By Lemma~\ref{lemma:CokernelN} we need to prove the existence of integral points such that the corresponding columns of $\alpha^*$ form a matrix of determinant $\pm n$. As in Section~\ref{sec:AlphaStarBasic} we split the proof into three cases.

\subsubsection{Basic matrix algebra}
Let $I_k$ be the identity matrix, $R_{k}$ the sparse matrix whose first row contains entirely of $1$'s, $S_{k}$ the sparse matrix whose lower diagonal consists of $1$'s ($S_1=0$), and $T_k$ the sparse matrix whose lower right entry is $1$. The index $k$ denotes that the matrices are $k\times k$.
\begin{equation}
R_3=\begin{pmatrix}1&1&1\\&&\\&&\end{pmatrix},\quad S_3=\begin{pmatrix}&&\\1&\\&1\end{pmatrix},\quad T_3=\begin{pmatrix}&&\\&&\\&&1\end{pmatrix}
\end{equation}
\begin{lemma}\label{lemma:DeterminantIdentity} We have
\begin{equation}
\det(I_k+R_k-S_k)=k+1,\qquad \det(I_k+R_k+T_k-S_k)=2k+1.
\end{equation}
\end{lemma}
\begin{proof}
This follows e.g.~by expanding the determinant using the last column. The matrices $I_k+R_k-S_k$ are shown below for $k=1,2,3$ and $4$.
\begin{equation}
\begin{pmatrix}2\end{pmatrix},\qquad \begin{pmatrix}2&1\\-1&1\end{pmatrix},\qquad \begin{pmatrix}2&1&1\\-1&1&\\&-1&1\end{pmatrix},\qquad\begin{pmatrix}2&1&1&1\\-1&1&&\\&-1&1&\\&&-1&1\end{pmatrix}
\end{equation}
For $I_k+R_k+T_k-S_k$, the only difference is that the lower right entry is now $2$.
\end{proof}
\begin{lemma}\label{lemma:BlockDeterminant}
Let $A$, $B$, $C$, $D$ be $k\times k$, $k\times l$, $l\times k$, and $l\times l$ matrices, respectively, and let $M=\left(\begin{smallmatrix}A &B\\C&D\end{smallmatrix}\right)$. If $D$ is invertible, we have
\begin{equation}
\pushQED{\qed}
\det(M)=\det(D)\det(A-BD^{-1}C).\qedhere\popQED
\end{equation}
\end{lemma}
\begin{proof}
This follows from the identity $\left(\begin{smallmatrix}A &B\\C&D\end{smallmatrix}\right)=\left(\begin{smallmatrix}I &B\\0&D\end{smallmatrix}\right)\left(\begin{smallmatrix}A-BD^{-1}C &0\\D^{-1}C&I\end{smallmatrix}\right)$.
\end{proof}
  
\subsubsection{One cusp}
Pick any face of $\T$ and consider the integral points shown in Figures~\ref{fig:OneCuspBasicOdd} and \ref{fig:OneCuspBasicEven}. Let $A_n$ be the $(n-1)\times(n-1)$ matrix formed by the corresponding columns of $\alpha^*$. The columns are ordered as shown in the figures, and the rows, i.e. the generators $x\otimes e_i$ of $C^n_0$, are ordered in the natural way (increasing in $i$). The following is an immediate consequence of the definition of $\alpha^*$.
\begin{lemma}
The matrix $A_n$ is given by
\begin{equation}
A_{2k+1}=\begin{pmatrix}I_k+R_k+T_k&I_k\\S_k&I_k\end{pmatrix},\qquad A_{2k}=\begin{pmatrix}2&0\cdots01&0\\0&I_{k-1}+R_{k-1}&I_{k-1}\\0&S_{k-1}&I_{k-1}\end{pmatrix}
\end{equation}
\end{lemma}
\begin{corollary}
The determinant of $A_n$ is $\pm n$.
\end{corollary}
\begin{proof}
This follows from Lemma~\ref{lemma:BlockDeterminant} and Lemma~\ref{lemma:DeterminantIdentity}.
\end{proof}

\begin{figure}[htpb]
\centering
\begin{minipage}[c]{0.48\textwidth}
\scalebox{0.8}{\input{figures_gen/OneCuspBasicOdd.tex}}
\captionsetup{width=10cm}
\caption{Basic generators, $n=2k+1$.}\label{fig:OneCuspBasicOdd}
\end{minipage}		
\begin{minipage}[c]{0.48\textwidth}
\scalebox{0.8}{\input{figures_gen/OneCuspBasicEven.tex}}
\captionsetup{width=10cm}
\caption{Basic generators, $n=2k$.}\label{fig:OneCuspBasicEven}
\end{minipage}				
\end{figure}

\subsubsection{Multiple cusps, self edges}
Pick a face with a self edge, and extend to a maximal tree with $1$-cycle $G$ as in~Figure~\ref{fig:TreeWithCycle}. Let $T=G\setminus\varepsilon_1$, and let   
$B_n$ denote the matrix formed by the columns of $\alpha^*$ corresponding to the face points shown in Figure~\ref{fig:OneCycleBasic} together with the edge points on $T$.
We order the generators $x_i\otimes e_j$ of $C_0^n$ as follows
\begin{equation}
x_1\otimes e_1,\dots, x_1\otimes e_{n-1},\quad x_2\otimes e_{n-1},\dots ,x_2\otimes e_1
\end{equation}
with a similar scheme for the other vertices. The following is an immediate consequence of the definition of $\alpha^*$.

\begin{lemma}
The matrix $B_n$ is given by
\begin{equation}
B_n=\left(\begin{array}{c|c}
I_{n-1}+R_{n-1}& \\
S_{n-1}&I_T\otimes\Z^{n-1}\\\cline{1-1}
 0 &
\end{array}\right),
\end{equation}
where $I_T\otimes\Z^{n-1}$ is the matrix obtained from $I_T$ by replacing each non-zero entry by $I_{n-1}$.\qed
\end{lemma}
\begin{corollary}
The determinant of $B_n$ is $\pm n$.
\end{corollary}
\begin{proof}
This follows from
\begin{equation}
\det(B_n)=\pm\det(\begin{pmatrix}I_{n-1}+R_{n-1}&I_{n-1}\\S_{n-1}&I_{n-1}\end{pmatrix}=\pm n,
\end{equation}
where the second equality follows from Lemmas~\ref{lemma:BlockDeterminant} and \ref{lemma:DeterminantIdentity}.
\end{proof}
\begin{figure}[htpb]
\centering
\begin{minipage}[c]{0.48\textwidth}
\scalebox{0.8}{\input{figures_gen/OneCycleBasic.tex}}
\captionsetup{width=10cm}
\caption{Basic generators, tree with $1$-cycle.}\label{fig:OneCycleBasic}
\end{minipage}	
\begin{minipage}[c]{0.48\textwidth}
\scalebox{0.8}{\input{figures_gen/ThreeCycleBasic.tex}}
\captionsetup{width=10cm}
\caption{Basic generators, tree with $3$-cycle.}\label{fig:ThreeCycleBasic}
\end{minipage}					
\end{figure}

\subsubsection{Multiple cusps, no self edge}
Pick a maximal tree with $3$-cycle $G$, and let $C_n$ be the matrix formed by the columns of $\alpha^*$ corresponding to the face points in Figure~\ref{fig:ThreeCycleBasic} together with the edge points on $T=G\setminus\varepsilon_1$.
\begin{lemma}
The matrix $C_n$ is given by
\begin{equation}
C_n=\left(\begin{array}{c|c}
I_{n-1}& \\
S_{n-1}&I_T\otimes\Z^{n-1}\\
R_{n-1}&\\\cline{1-1}
 0 &
\end{array}\right),
\end{equation}
\end{lemma}
\begin{corollary}
The determinant of $C_n$ is $\pm n$.
\end{corollary}
\begin{proof}
We have
\begin{equation}
\det(C_n)=\pm\det(M),\qquad M=\begin{pmatrix}I_{n-1}&I_{n-1}&\\S_{n-1}&I_{n-1}&I_{n-1}\\R_{n-1}&&I_{n-1}\end{pmatrix}.
\end{equation}
Using Lemma~\ref{lemma:BlockDeterminant} with $A=\left(\begin{smallmatrix}I_{n-1}&I_{n-1}\\S_{n-1}&I_{n-1}\end{smallmatrix}\right)$, $B=\left(\begin{smallmatrix}0\\I_{n-1}\end{smallmatrix}\right)$, $C=\left(\begin{smallmatrix}R_{n-1}&0\end{smallmatrix}\right)$, $D=I_{n-1}$, we have
\begin{equation}
\det(M)=\det\begin{pmatrix}I_{n-1}&I_{n-1}\\S_{n-1}-R_{n-1}&I_{n-1}\end{pmatrix}=\det(I_{n-1}+R_{n-1}-S_{n-1})=n,
\end{equation}
where the second equation follows from Lemma~\ref{lemma:BlockDeterminant}, and the third from Lemma~\ref{lemma:DeterminantIdentity}.
\end{proof}
This concludes the proof that $\alpha^*$ is basic.

\bibliographystyle{plain}
\bibliography{BibFile}

\end{document}